\documentclass[11pt]{amsart}

\usepackage{amssymb,amsmath,amsthm} 
\usepackage{stmaryrd} 

\usepackage{enumitem} 
\setenumerate{label=(\alph*), font=\normalfont}

\usepackage[all]{xy} 

\usepackage{multirow} 
\usepackage{tabu}
\usepackage{xcolor}

\theoremstyle{plain}
\newtheorem{theorem}{Theorem}[section]
\newtheorem{lemma}[theorem]{Lemma}
\newtheorem{proposition}[theorem]{Proposition}

\theoremstyle{definition}

\newtheorem{remark}[theorem]{Remark}

\numberwithin{equation}{section}

\begin{document}

\title[Asymptotics of immaculate line bundles]{Asymptotics of immaculate line bundles on smooth toric Deligne-Mumford stacks}

\author{Lev Borisov}
\address{Department of Mathematics\\
Rutgers University\\
Piscataway, NJ 08854} \email{borisov@math.rutgers.edu}

\author{Alexander Duncan}
\address{Department of Mathematics\\
University of South Carolina\\
Columbia, SC 29208} \email{duncan@math.sc.edu}

\begin{abstract}
A line bundle is immaculate if its cohomology vanishes in every dimension.
We give a criterion for when a smooth toric Deligne-Mumford stack has
infinitely many immaculate line bundles. This answers positively a
question of Borisov and Wang.
As a byproduct, we describe the asymptotic behaviour of the collection
of immaculate line bundles.
\end{abstract}



\maketitle


\section{Introduction}

Let $X$ be a smooth proper toric Deligne-Mumford stack of dimension $d$
over an algebraically closed field $k$ of characteristic $0$.
In particular, $X$ could be a smooth proper toric variety.
Following \cite{ABKW}, a line bundle $\mathcal{L}$ on $X$ is called \emph{immaculate} iff
$H^i(X,\mathcal{L})=0$ for all $i \ge 0$. We denote the set of isomorphism classes of immaculate line bundles
by  $\operatorname{Imm}(X)$. Following work in \cite{Wang,BorisovWang}, 
we find a criterion for when $\operatorname{Imm}(X)$ is infinite.

\smallskip
To describe the criterion, we need to use the notion of \emph{forbidden cones}, introduced in ~\cite{BorisovHua}. 
These are rational polyhedral cones $F_I$ in the real vector space
$\operatorname{Pic}_{\mathbb{R}}(X)
:= \operatorname{Pic}(X) \otimes_{\mathbb{Z}} \mathbb{R}$ indexed by certain subsets $I$
of the rays in the fan of $X$; see Section~\ref{sec:forbidden} for more details.
If the image in $\operatorname{Pic}_{\mathbb{R}}(X)$ of a line bundle
does not lie in any of the forbidden cones, then it is immaculate (although the converse generally fails).
Let $C_I$ denote the translate of the forbidden cone $F_I$ such that the
vertex is at the origin.

\smallskip
In this paper we prove the following statement, which answers positively one of the questions raised in \cite[Section 5]{BorisovWang}.
\begin{theorem} \label{thm:infinity_condition}
There exist infinitely many immaculate line bundles on $X$ if and only
if there exists a line $l$ passing through the origin of
$\operatorname{Pic}_{\mathbb{R}}(X)$
such that, for every translated forbidden cone $C_I$ of $X$,
the intersection with the interior $l\cap C^\circ_I$
is empty.
\end{theorem}

It is easy to see that the union of all translated forbidden cones $C_I$
is all of $\operatorname{Pic}_{\mathbb{R}}(X)$.
The theorem above is equivalent to saying that there are finitely many
immaculate line bundles if and only if the interiors $C^\circ_I$
of all the translated forbidden cones cover
$\operatorname{Pic}_{\mathbb{R}}(X) \setminus \{0\}$.

In the course of proving this theorem, we actually establish a finer
structural result on the asymptotics of immaculate line bundles, which
we describe now. 

\smallskip
We embed $\operatorname{Pic}_{\mathbb{R}}(X) \cong \mathbb{R}^r$
as an open subset of a real projective space $\mathbb{P}(\operatorname{Pic}_{\mathbb{R}}(X)\oplus \mathbb R)\cong \mathbb{RP}^r$
where $r$ is the rank of $\operatorname{Pic}(X)$.
Let $\operatorname{Imm}^\infty(X)$ denote the set of accumulation points
of the image of $\operatorname{Imm}(X)$ in $\mathbb{RP}^r$.
Clearly,  $\operatorname{Imm}^\infty(X)$ is a subset of
the hyperplane at infinity $\Pi \cong \mathbb{RP}^{r-1}$.
For each forbidden cone, let $D_I := \overline{C_I} \setminus C_I$
be the complement of the closure of $C_I$ in $\mathbb{RP}^r$. We see that $D_I$ is a closed subset of $\Pi$.

\begin{theorem} \label{thm:main}
Let $D^\circ_I$ be the relative interior of $D_I$ in $\Pi$.
Then
\[\operatorname{Imm}^\infty(X) = \Pi \setminus \bigcup_{I} D^\circ_I\]
where the union is over all $I$ for which $C_I$ is a translate of a forbidden cone. 
\end{theorem}

\begin{remark}
It is clear that Theorem  \ref{thm:infinity_condition} is an immediate consequence of Theorem \ref{thm:main}.
Indeed, the torsion of $\operatorname{Pic}(X)$ is finite, so $\operatorname{Imm}(X)$ is infinite if and only if its image in $\operatorname{Pic}_{\mathbb R}(X)$ is infinite. The latter is infinite if and only if it has an accumulation point in $\mathbb {RP}^r$.
\end{remark}

\subsection*{Acknowledgements}
This work originated at the workshop
\emph{Syzygies and mirror symmetry}
held at the American Institute of Mathematics. 
A.~Duncan was supported by a grant from the Simons Foundation (638961).
We also thank Joseph Helfer for useful conversations
and Markus Perling for helpful suggestions.

\section{Forbidden Cones}
\label{sec:forbidden}

We review the notion of forbidden cones and their relation to
immaculate line bundles first established in \cite{BorisovHua}.
Note that our conventions are more in line with \cite{ABKW}.
 
\smallskip
Recall from \cite{BCS} that a smooth projective toric Deligne-Mumford
stack $X$ is determined by the combinatorial data of a 
\emph{stacky fan}
$\mathbf{\Sigma} = (N, \Sigma, \{v_i\})$.
Here $N$ is a finitely generated free abelian group,
$\Sigma$ is a complete simplicial fan in
$N \otimes_{\mathbb{Z}} \mathbb{R}$,
and $v_1,\ldots,v_n$ are nonzero elements of $N$, one in each of the rays of
$\Sigma$.
The rank of $N$ is the same as the dimension $d$ of $X$.

\begin{remark}
More generally, one can take $N$ to be any finitely generated abelian group.
For simplicity of presentation, we assume that $N$ is \emph{free}.
\end{remark}

\smallskip
The Picard group $\operatorname{Pic}(X)$ is
isomorphic to the quotient of the group $\mathbb{Z}^n$ of torus-equivariant divisors by the subgroup of principal torus-equivariant divisors.
More precisely, the rays of $\Sigma$ give rise to a canonical generating set
$E_1,\ldots, E_n$ for $\operatorname{Pic}(X)$. The relations on $\{E_i\}$ are given by 
$$
\sum_{i=1}^n \langle m, v_i\rangle E_i = 0
$$
for all $m\in M = \operatorname{Hom}(N,\mathbb Z)$.
We denote by $\bar E_i$ the images of $E_i$ in $\operatorname{Pic}_{\mathbb R}(X)$.
The relations among $\bar E_i$  are given by the same formula with $m\in M_{\mathbb R}$.

\smallskip
Given a subset $I \subseteq \{1,\ldots,n\}$, we define a
subset $FS_I = q_I + CS_I$ of $\operatorname{Pic}(X)$ where
\[
q_I = - \sum_{i \in I} E_i,
\]
is a point and
\[
CS_I = \sum_{i \notin I} \mathbb{Z}_{\ge 0} E_i - \sum_{i \in I}
\mathbb{Z}_{\ge 0} E_i,
\]
is a monoid.
More precisely, we define a function
$\pi_I : \mathbb{Z}_{\ge 0}^n \to \operatorname{Pic}(X)$ via
\begin{equation}\label{pi}
\pi_I(a_1,\ldots,a_n) = \sum_{i \notin I} a_i E_i - \sum_{i \in I}
(1+a_i) E_i
\end{equation}
and observe that $FS_I$ is the image of $\pi_I$.
We also define a convex cone $F_I := q_I + C_I$ in
$\operatorname{Pic}_{\mathbb{R}}(X)$ where
\[
C_I = \sum_{i \notin I} \mathbb{R}_{\ge 0}\bar E_i - \sum_{i \in I}
\mathbb{R}_{\ge 0}\bar E_i .
\]
Clearly, the image of $FS_I$ under $\operatorname{Pic}(X)\to \operatorname{Pic}_{\mathbb R}(X)$ lies in $F_I$.

\smallskip
Recall that the fan $\Sigma$ gives rise to an abstract simplicial
complex on $\{1,\ldots,n\}$, which we also denote by $\Sigma$.
Given a subset $I \subseteq \{1,\ldots,n\}$,
we may form a new simplicial complex $\Sigma|_I$ on $I$
where $\sigma \subseteq I$ is in $\Sigma|_I$ if and only if
$\Sigma$ contains the cone generated by the rays indexed by $\sigma$.
Observe that $\Sigma|_{\{1,\ldots,n\}}$ corresponds to $\Sigma$ itself,
while $\Sigma|_\emptyset$ is the simplicial complex $\{\emptyset\}$.
Let $\widetilde{H}_n(\Omega,k)$ denote the reduced homology of the
abstract simplicial complex $\Omega$.
Recall that $\widetilde{H}_{-1}(\Omega,k) \ne 0$
if and only if $\Omega=\{\emptyset\}$.

\begin{proposition} \label{prop:coh_bound} 
For $\mathcal{L} \in \operatorname{Pic}(X)$, we have
\[
H^i(X,\mathcal{L}) = \bigoplus_{I \subseteq \{ 1, \ldots, n\}}
\widetilde{H}_{i-1}(\Sigma|_I,k)^{\oplus p_I} .
\]
where $p_I$ is the cardinality of $\pi_I^{-1}(\mathcal{L})$.
\end{proposition}

\begin{proof}
Adjusting for conventions,
this is \cite[Proposition 4.1]{BorisovHua}.
\end{proof}

If the reduced homology $\widetilde{H}_i(\Sigma|_I,k)$
is not trivial for some $i \ge -1$, then
we say $I$ is \emph{tempting}.
If $I$ is tempting, then $FS_I$ is a \emph{forbidden set}
and $F_I$ is a \emph{forbidden cone}.
As an immediate consequence of the previous proposition we have:

\begin{proposition} \label{prop:forbidden_immaculate}
A line bundle $\mathcal{L}$ is immaculate if and only if $\mathcal{L}$ is not
contained in any forbidden set $FS_I$ for a tempting set $I$. 
\end{proposition}

\begin{remark}\label{immcond}
Proposition \ref{prop:forbidden_immaculate} implies that 
if the image of $\mathcal{L}$ in $\operatorname{Pic}_{\mathbb{R}}(X)$
is not contained in any forbidden cone, then $\mathcal{L}$ is immaculate.
However, the converse does not generally hold. 
\end{remark}

Our conventions ensure that $\emptyset$ is tempting and $F_\emptyset$
is just the effective cone.
Observe that $\Sigma|_{\{1,\ldots,n\}}$ corresponds to $\Sigma$ itself,
which is a simplicial $(d-1)$-sphere; thus
$\widetilde{H}_{d-1}(\Sigma) \ne 0$ and $\{1,\ldots,n\}$ is tempting.
The forbidden set $FS_{\{1,\ldots,n\}}$ corresponds to line bundles $\mathcal{L}$
for which $H^d(X,\mathcal{L}) \ne 0$.

\begin{remark} \label{rem:SerreDuality}
Let $s : \operatorname{Pic}(X) \to \operatorname{Pic}(X)$ be the Serre
duality map $s(D)=K_X - D$.
The canonical divisor of $X$ is given by
\[
K_X = q_{\{1,\ldots,n\}} = - \sum_{i = 1}^n E_i .
\]
Thus $s(FS_I)=FS_{I^c}$ for all subsets $I$
where $I^c := \{1,\ldots,n\} \setminus I$ denotes the complement.
Consequently, $I$ is tempting if and only if $I^c$ is tempting.
(See also \cite[Remark 5.4]{ABKW}.)
\end{remark}

\section{The Thomsen Zonotope}

In this section, we introduce the \emph{Thomsen Zonotope};
this is a canonical bounded region of
$\operatorname{Pic}_{\mathbb{R}}(X)$
that is an important source of immaculate line bundles.
The geometry of this region is closely related to that of the
forbidden cones.

For every positive integer $m$, the $m$-th \emph{toric Frobenius}
is the endomorphism $F_m: X \to X$ obtained by extending the power map
$x \mapsto x^m$ on the torus to all of $X$.
In \cite{Thomsen}, Thomsen shows that the direct image
$\left(F_m\right)_\ast \mathcal{O}_X$ of the structure sheaf is a direct sum
of line bundles and gave a precise description of the line bundles that
result (see \cite{Achinger} for a shorter proof).
The \emph{Thomsen collection} is the (finite) set of line bundles that
occur as direct summands of
\[\bigoplus_{m=1}^\infty\left(F_m\right)_\ast \mathcal{O}_X . \]
It is well known that all non-trivial line bundles in the Thomsen
collection are immaculate.

\smallskip
We define the \emph{Thomsen half-open zonotope} as the set
\[
Z^h := \left\{ \sum_{i=1}^n \gamma_i \bar E_i\ \middle|\
\gamma_i\in (-1,0] \right\} ,
\]
in $\operatorname{Pic}_\mathbb{R}(X)$.
It is a convex set, and 
the line bundles in the Thomsen collection map to $Z^h$. 
We denote the closure of $Z^h$ by $Z$ and the interior by $Z^\circ$.
The polytope $Z$ is a \emph{zonotope} --- a Minkowski sum of finitely many line
segments $[-1,0] \bar E_i$. Similarly, $Z^\circ = \sum_{i=1}^n (-1,0) \bar E_i$.

\begin{proposition} \label{prop:forbiddenZonotope}
If $F_I = q_I + C_I$ is a forbidden cone,
then $Z$ is a subset of $q_I-C_I$.
Moreover, $Z \cap U = (q_I-C_I) \cap U$ in a neighborhood $U$ of $q_I$.
\end{proposition}

\begin{proof}
This follows almost immediately from the definitions.
The set $Z$ consists of points that can be written as
\[
\sum_{i =1}^n a_i\bar E_I
\]
where $a_i \in [-1,0]$ for all indices.
The set $q_I-C_I$ consists of such points where $a_i \le 0$
if $i \notin I$ and $a_i \ge -1$ if $i \in I$.
The points in $Z$ are all of this form.
Conversely, for a sufficiently small $\varepsilon$-ball $U$ around $q_I$,
there exists a $0 < \delta < 1$, such that
all points of $(q_I - C_I) \cap U$ can be written as above
with $a_i \in (-\delta,0]$ for $i \notin I$ and $a_i \in [-1,-1+\delta)$
for $i \in I$.
These points lie inside $Z$, thus the second statement follows.
\end{proof}

\begin{lemma} \label{lem:Zinterior_nonempty}
There is a line bundle mapping to $Z^\circ$.
\end{lemma}

\begin{proof}
Pick an element $m\in \operatorname{Hom}(N,\mathbb R)$ such that $\langle m,v_i\rangle$ is noninteger for all $i$.
Then 
$$
\sum_{i=1}^n \lfloor{\langle m,v_i\rangle}\rfloor \bar E_i = -\sum_{i=1}^n  \{\langle m,v_i\rangle\} \bar E_i \in Z^\circ
$$
so the line bundle $\sum_{i=1}^n \lfloor{\langle m,v_i\rangle}\rfloor  E_i$ has the desired property. 
\end{proof}

The next result is not new but we provide a proof for the reader's benefit.
A similar statement can be found in \cite[Lemma~5.24(ii)]{ABKW}.
\begin{lemma} \label{lem:CIstrongConvex}
If $I$ is tempting, then $q_I$ is a 
vertex of $Z$.
Equivalently, $C_I$ is strongly convex with vertex at the origin.
\end{lemma}

\begin{proof}
From Proposition~\ref{prop:forbiddenZonotope},
we see that $q_I$ is a vertex of $Z$ if and only if
$C_I \cap (-C_I) \cap U = \{0\}$ for some neighborhood $U$
of the origin.

Suppose $q_I$ is not a vertex of $Z$.
Then there is a non-zero vector $v$ in $C_I \cap (-C_I)$.
Since $C_I$ is rational polyhedral we can assume $v$ can be written as a  $\mathbb Q$-linear combinations of the generators of both cones. By taking the difference and scaling, we obtain a non-trivial
relation
$
\sum_{i \notin I} a_i\bar E_i - \sum_{i \in I} a_i\bar E_i = 0
$
in $\operatorname{Pic}_{\mathbb{R}}(X)$, 
where integers $a_1,\ldots, a_n$ are nonnegative and not all zero.  We can scale further by an appropriate
integer to ensure that
\[
\sum_{i \notin I} a_i E_i - \sum_{i \in I} a_i E_i = 0
\]
holds in 
$\operatorname{Pic}(X)$.
For any nonnegative integer $k$ we write
$$-\sum_{i\in I} E_i = \sum_{i\in I} (-1 - k a_i) E_i + \sum_{i\not\in I} (k a_i) E_i$$
so $\pi_I^{-1}(-\sum_{i\in I} E_i)$ is infinite.
Since $I$ is tempting, by Proposition~\ref{prop:coh_bound} we get $H^*(-\sum_{i\in I} E_i)$ is infinite-dimensional, 
in contradiction to $X$ being proper.
\end{proof}

\begin{remark}
There are typically vertices of $Z$ which do not correspond to any tempting sets $I$.
\end{remark}

\section{Proof of the Main Theorem}

We will use the following well known result,
which is often considered in the context of
higher-dimensional analogs of the Frobenius Coin Problem;
see, for example, \cite[Theorem 6.5.1]{Alfonsin}.

\begin{lemma} \label{lem:FrobeniusCoin}
Let $L$ be a finitely generated abelian group and let $\pi:L\to L_{\mathbb R}$ be the natural map.
Let $M$ be a finitely generated submonoid of $L$ which generates $L$ as a subgroup.
Then there exists an element $m \in M$ such that 
\[
\pi^{-1}
(m + \mathbb{R}_{\ge 0}\pi(M))  \subseteq M.
\]
\end{lemma}

We now proceed to the proof of our main theorem.

\begin{proof}[Proof of Theorem~\ref{thm:main}]
The argument for the inclusion $\subseteq$ in Theorem~\ref{thm:main}
is essentially in the proof of \cite[Proposition 3.8]{Wang}.
For the reader's convenience, we sketch it here.

Suppose $I$ is tempting.
Observe that $\Pi \cap \overline{v+C_I} = \Pi \cap \overline{v'+C_I}$
for any $v,v' \in \operatorname{Pic}_{\mathbb{R}}(X)$
and $\Pi \cap \overline{C_I} = \Pi \cap \overline{C_{I^c}}$.
In view of Proposition~\ref{prop:forbidden_immaculate} and
Lemma~\ref{lem:FrobeniusCoin} applied to the group $ \operatorname{Pic}(X)$ and monoids $CS_I$
and $CS_{I^c}$,
there exist $v,v' \in \operatorname{Pic}(X)$
such that $v+C_I$ and $v'+C_{I^c}$ 
contain no images of immaculate line bundles under the map
$\operatorname{Pic}(X)\to\operatorname{Pic}_{\mathbb{R}}(X)$.
It remains to observe that for any $L\in D_I^\circ$ the set
$A = \overline{(v+C_I) \cup (v'+C_{I^c})}$
contains an open neighborhood of $L$
and no images of immaculate line bundles.
Thus $L$ is not in $\operatorname{Imm}^\infty(X)$.

\smallskip
We now establish the inclusion $\supseteq$, which is new.

\smallskip
Suppose $L \in \Pi$ is a point such that
$L \notin D^\circ_I$ for any tempting $I$. We first assume that 
$L$ is rational, so that the corresponding line
$\ell \subseteq \operatorname{Pic}_{\mathbb{R}}(X)$ has infinitely many 
images of elements of $\operatorname{Pic}(X)$ on it.
Pick a point $z_0$ in $Z^\circ$ which is an image of a line bundle, its existence
guaranteed by Lemma~\ref{lem:Zinterior_nonempty}.
Since $\ell$ is rational, there are infinitely many elements $\bar E$ in
the shifted line $z_0+\ell$ which are images of 
elements of $\operatorname{Pic}(X)$. These have $L\in \Pi$ as their (unique) accumulation point, and we will show that
each $\bar E$ is an image of an immaculate line bundle, thus establishing that $L\in \operatorname{Imm}^\infty(X)$.
By  Proposition \ref{prop:forbidden_immaculate} it suffices to show that
$z_0 + \ell$ is disjoint from all $F_I$ for tempting $I$.

\smallskip
By Lemma~\ref{lem:CIstrongConvex}, 
$C_I$ is a strongly convex rational polyhedral cone with
vertex at the origin.
Since $\ell$ is disjoint from the interior of $C_I$, there exists a nonzero linear functional $h:\operatorname{Pic}_{\mathbb R}(X)\to \mathbb R$ such that $h(C_I) \ge 0$
but $h(\ell) = 0$. 
By Proposition~\ref{prop:forbiddenZonotope},
we see that $z_0-q_I \in -C_I^\circ$, thus $h(z-q_I) < 0$. So for any
$v\in \ell$ we have $h(z_0+v) < h(q_I)$, while we have
$h(F_I) = h(q_I) + h(C_I)\geq h(q_I)$. 

\smallskip
We have thus established that the rational points of
\[
\Pi \setminus \bigcup_{I \textrm{ tempting}} D^\circ_I
\]
lie in $\operatorname{Imm}^\infty(X)$.
However, there are finitely many inequalities with rational coefficients
defining each $C_I$.
Moreover, there are only finitely many
tempting sets $I$.
Thus, the rational points of the set
$\Pi \setminus \bigcup_{I \textrm{ tempting}} D^\circ_I$
are dense.
Since all rational points lie in $\operatorname{Imm}^\infty(X)$
and both sets are closed, we have the desired inclusion.
\end{proof}

\begin{remark}
Our main theorem is phrased using the compactification of
$\operatorname{Pic}_{\mathbb{R}}(X)$ with boundary
$\Pi \cong \mathbb{RP}^{r-1}$.
Instead, we could have used the spherical boundary
$\left( \operatorname{Pic}_{\mathbb{R}}(X) \setminus \{0\}\right)
/ \mathbb{R}_{> 0} \cong S^{r-1}$.
However, in view of Remark~\ref{rem:SerreDuality},
there is no substantial difference between these two perspectives.
\end{remark}

\begin{remark}
From the proof of our main theorem, we see that if there are infinitely
many immaculate line bundles, then there are infinitely many outside the
forbidden cones.
However, there may still be infinitely many line bundles \emph{inside}
the forbidden cones as well.
For example, consider the stack $X=\mathbb{P}(2:3) \times \mathbb{P}^1$
where $\mathbb{P}(2:3)$ is a weighted projective line.
The line bundles $\mathcal{O}_X(1,n)$ are immaculate for every positive $n$,
but they all lie in the interior of the effective cone.
\end{remark}

\begin{remark}
In \cite{Perling}, Perling explores conditions for vanishing of
cohomology groups in the closely related context of
divisorial sheaves on toric varieties.
In particular, he finds a sufficient condition for the existence of
infinitely many immaculate divisorial sheaves in terms of the Iitaka
dimensions of divisors on the boundary of the nef cone.
There, the study of \emph{discriminantal arrangements} plays a similar
role as our use of the Thomsen Zonotope above.
\end{remark}

\begin{remark}
Finally, we point out that while the criterion from
Theorem~\ref{thm:main} is computable in theory,
finding all the tempting subsets and forbidden cones
is complicated in practice.
Another condition that can be checked directly on the fan
is conjectured in \cite[Section 5]{BorisovWang}.

An element of $\operatorname{Pic}_{\mathbb{R}}(X)$ can be represented as
a function $\psi$ on $N_{\mathbb{R}}$ which is linear on every cone of
$\Sigma$, which is well-defined up to a global linear function.
The maximal cones $\{\sigma\}$ of $\Sigma$ then correspond to a finite set of points
$\{\psi_\sigma\}$ in the dual space $M_{\mathbb{R}}$.
If the convex hull $\Lambda_\psi$ of these points is not
full-dimensional, then we obtain a point of $\operatorname{Imm}^\infty(X)$
by \cite[Theorem 2.14]{BorisovWang}.

While this is a sufficient condition for having infinitely many immaculate line
bundles, it is not known whether it is necessary --- even for toric varieties
of dimension $3$.
\end{remark}

\end{document}